\documentclass{amsart}
\usepackage[utf8]{inputenc}
\usepackage{amscd}
\usepackage{amssymb}
\usepackage{pstricks}
  \usepackage{todonotes}
\usepackage{comment}
\usepackage{empheq}
\usepackage{graphicx}
 \newtheorem{theorem}{Theorem}
 
 \newtheorem{lemma}{Lemma}
 \newtheorem{definition}{Definition}

\newtheorem{remark}{Remark}

\DeclareMathOperator{\rank}{Rank}

\numberwithin{equation}{section}
\numberwithin{lemma}{section}
\numberwithin{definition}{section}
\numberwithin{lemma}{section}
\numberwithin{theorem}{section}

\author{Ignacio Huerta}

\address{Departamento de Matem\'atica, Universidad T\'ecnica Federico Santa Mar\'ia, Casilla 110-V, Valpara\'iso, Chile.}
\email{ignacio.huertan@usm.cl}

\title[Nonuniform complete observability]{Nonuniform complete observability; preservation by output feedback and duality results}
\thanks{This research has been partially supported by the grant FONDECYT Regular 1210733}

\begin{document}

\begin{abstract}
In this paper we propose a new observability property for nonautonomous linear control systems in finite dimension; the nonuniform complete observability, which is more general than the uniform complete observability. The main result of this work will prove that nonuniform complete observability is preserved via output feedback. In addition, the duality between this concept and the recently introduced concept of nonuniform complete controllability will be proved, leading to the result of preservation of nonuniform complete controllability via input feedback.
\end{abstract}

\maketitle

\section{Introduction}
\subsection{State of art} 
\textit{Observability} is a key concept in dynamical systems control theory, and was introduced by Rudolf E. Kalman in the 1960s. Kalman, best known for his pioneering work in the development of the Kalman filter \cite{Kalman1960} and optimal control theory \cite{Kalman}, also laid the foundation for the modern analysis of linear and nonlinear systems, where the concept of observability plays a fundamental role.

The observability is related to the ability to reconstruct the internal state of a control system from its outputs in a finite time interval.
Roughly speaking, in a control system, the initial state $x_0=x(t_0)\in\mathbb{R}^{n}$ is \textit{observable at time $t_{0}\geq 0$} if this state can be uniquely determined from the knowledge of the output in the interval $t\in[t_0,t_1]$. In addition, a control system is i) \textit{observable at time $t_{0}\geq 0$} if any initial state $x_{0}=x(t_0)$ is observable at time $t_{0}\geq 0$ and ii) \textit{completely observable} (CO) if $x_0$ is observable at any time $t_{0}\geq 0$. 

The study of observability can present certain similarities and differences between \textit{linear time-invariant} (LTI) systems and \textit{linear time-varying} (LTV) systems, which we will detail below.

\subsubsection{\textnormal{\textbf{LTI perspective}}}
In formal terms, let us consider a LTI system described by the equations:
\begin{subequations}
  \begin{empheq}[left=\empheqlbrace]{align}
    & \dot{x}(t)=Ax(t)+Bu(t), \label{LTIsystem1a} \\
    & y(t)=Cx(t)+Du(t), \label{LTIsystem1b}
  \end{empheq}
\end{subequations}
where $x(t)\in\mathbb{R}^{n}$ is the \textit{state} of the system, $u(t)\in\mathbb{R}^{p}$ is the \textit{input}, $y(t)\in\mathbb{R}^{m}$ is the \textit{output} and $A, B, C, D$ are matrices with order $n\times n$, $n\times p$, $m\times n$ y $m\times p$ respectively. By considering LTI systems (see \cite[Chapter 9]{Rugh}), one of the ways of verifying observability is by means of the algebraic Kalman criterion, which states that the system is observable if and only if the observability matrix:
$$\mathcal{O}=\left[\begin{array}{c}
C \\
C A \\
C A^2 \\
\vdots \\
C A^{n-1}
\end{array}\right]\in M_{nm\times n}(\mathbb{R})$$
has rank equal to $n$, which coincides with the number of states of the system. From this criterion, we can notice that the observability does not depend on the matrices $B$ and $D$, therefore, from here on, we will be talking about the observability of the system $(A,C)$.

In addition, it is possible to study the observability of the system \eqref{LTIsystem1a}-\eqref{LTIsystem1b} from the perspective of the associated Gramian matrix \cite[Chapter 9]{Rugh}, namely, the nonsingularity property for any value $t_1>t_0$ of the square matrix $W_{\mathcal{O}}(t_0,t_1)$ described by 
$$W_{\mathcal{O}}(t_0,t_1)=\displaystyle\int_{t_0}^{t_1}e^{A^{T}(s-t_0)}C^{T}Ce^{A(s-t_0)}\;ds,$$
thus completing two criteria that are equivalent to the observability property. 

On the other hand, a concept closely related to observability is controllability, which measures the ability to manipulate the state of the system, in a finite time interval, using external inputs. There is a duality between controllability and observability, which means that observability problems can be analyzed using controllability results, and vice versa, i.e., there are results associated with the controllability matrix and controllability gram matrix (see \cite{Kalman} and \cite{Rugh}) similar to the aforementioned results equivalent to observability. Based on the existing literature, this duality will be analyzed in the non-uniform and LTV systems context in which this work will be developed.

\subsubsection{\textnormal{\textbf{LTV framework}}} There are obvious differences in the observability analysis when comparing LTI and LTV systems, precisely because of the dependence on a $t$ parameter of the matrices that were previously considered constant. However, there are results that allow testing the observability of a system that develop ideas similar to those of LTI systems.   

Let us consider a LTV control system
\begin{subequations}
  \begin{empheq}[left=\empheqlbrace]{align}
    & \dot{x}(t)=A(t)x(t), \label{control1a} \\
    & y(t)=C(t)x(t), \label{control1b}
  \end{empheq}
\end{subequations}
where $x(t)\in\mathbb{R}^{n}$ is the \textit{state vector}, $y(t)\in\mathbb{R}^{m}$ is the \textit{output}, while that $t\mapsto A(t)\in M_{n\times n}(\mathbb{R})$ and $t\mapsto C(t)\in M_{m\times n}(\mathbb{R})$ are matrix valued functions for any $t\geq 0$, having orders  $n\times n$ and $m\times n$ respectively. In addition, 
$A(\cdot)$ and $C(\cdot)$ are measurable and bounded on finite intervals.
The transition matrix of \eqref{control1a} is denoted by $\Phi_{A}(t,s)$.

Following the same line as presented in the previous subsection, there are results that allow us to prove the observability of an LTV system. On the one hand (see \cite{Rugh}), 
suppose that $C(t)$ is a $q$ times continuously differentiable matrix and $A(t)$ is a $(q-1)$ times continuously differentiable matrix. Furthermore, defining 
$$
\mathcal{L}(t)=\left[\begin{array}{c}
\mathbf{L}_0(t) \\
\vdots \\
\mathbf{L}_q(t)
\end{array}\right], \quad\textnormal{where}\quad \left\{\begin{array}{l}
\mathbf{L}_0(t)=C(t) \\
\mathbf{L}_i(t)=\mathbf{L}_{i-1}(t) A(t)+\dot{\mathbf{L}}_{i-1}(t), \quad i=1, \ldots, q ,
\end{array}\right.
$$
then the LTV system \eqref{control1a}-\eqref{control1b} is observable on $\left[t_0, t_f\right]$ if for some $t_a \in\left[t_0, t_f\right]$, we have that $\rank(\mathcal{L}(t_a))=n$.

On the other hand, there exists a well known necessary and sufficient condition ensuring both observability at time $t_{0}$ and complete observability, which is stated in terms of the observability Gramian matrix, usually defined by
\begin{equation}
 \label{gramianobservability}
M(t_0,t_f)=\displaystyle\int_{t_0}^{t_f}\Phi_{A}^{T}(s,t_0)C^{T}(s)C(s)\Phi_{A}(s,t_0)\;ds.
\end{equation}

The control system \eqref{control1a}-\eqref{control1b} is {\it observable at time $t_{0}\geq 0$} if and only if there exists
$t_{f}>t_{0}\geq 0$ such that $M(t_{0},t_{f})>0$. In addition, it is {\it completely observable} if and only if for 
any $t_{0}\geq 0$, there exists $t_{f}>t_{0}$ such that $M(t_{0},t_{f})>0$. We refer the reader to \cite{Anderson67,Kreindler} for a detailed description. 

As in the context of time-invariant systems, we say without distinction that system \eqref{control1a}-\eqref{control1b} or the pair $(A,C)$ is completely observable. In this sense, this notation option can be applied to any type of definition associated with the observability concept.

A concept that is more restrictive than complete observability, and which in turn implies it, corresponds to the case where $t_0=t$ and $t_f=t+\sigma$, for a fixed $\sigma>0$, which corresponds to the \textit{uniform complete observability}, whose formal definition is given below (see \cite[See Definition 3.3.3]{Ioannou10.5555/211527} and \cite[Section 1.5]{sastry10.5555/63437} for more detailed references):

\begin{definition}    The LTV control system \eqref{control1a}-\eqref{control1b} is said to be uniformly completely observable \textnormal{(UCO)} on $[0,+\infty)$ if there exists a fixed constant $\sigma>0$ and positive numbers $\bar{\alpha}_0(\sigma)$ and $\bar{\alpha}_1(\sigma)$ such that the following relations hold for all $t\geq 0$: 
\begin{equation}
\label{UCO}
0<\bar{\alpha}_0(\sigma)I\leq M(t,t+\sigma)\leq \bar{\alpha}_1(\sigma)I.
\end{equation}
\end{definition}

It is usual to find in the literature (see \cite{BATISTA20173598}) that $A(\cdot)$ and $C(\cdot)$ are initially considered to be bounded matrices, therefore in \eqref{UCO} only requires 
$$\bar{\alpha}_0(\sigma)I\leq M(t,t+\sigma),$$
since the second inequality is obtained as a consequence of the bounding of the matrices $A(\cdot)$ and $C(\cdot)$.

\subsection{Novelty of this work}
In this paper we introduce and characterize the concept of non-uniform complete observability (NUCO), which is more general than the classical Kalman property of uniform complete observability, but more restrictive than complete observability.

The first novel result corresponds to the \ref{dualidad} Theorem, in which, based on the definition of complete non-uniform controllability, a close relationship between both properties is established, allowing to flow from one to the other and their respective consequences. 

The second novel result of this work is the main theorem of the paper (Theorem \ref{observableconK}), providing the preservation of nonuniform complete observability considering an output feedback. An additional novel consequence of this result is the use of duality to prove that complete nonuniform controllability is preserved by input feedback.

\subsection{Notations}
Given $M\in M_{n}(\mathbb{R})$, $M^{T}$ is the transpose. The inner product of two vectors $x,y\in \mathbb{R}^{n}$ is denoted by
$\langle y, x\rangle=x^{T}y$ and the euclidean norm of a vector will be denoted by $|x|=\sqrt{\langle x,x\rangle}$. 

A symmetric matrix $M=M^{T}\in M_{n}(\mathbb{R})$ is semi--positive definite if
$x^{T}Mx =\langle M x, x\rangle  \geq 0$ for any real vector $x\neq 0$, and this property will be denoted as $M \geq 0$.
In case that the above inequalities are strict, we say that $M=M^{T}$ is positive definite.

Given two matrices $M, N\in M_{n}(\mathbb{R})$, we write $M\leq N$ if  $N-M\geq0$ or equivalently, if $\langle M x, x\rangle \leq\langle N x, x\rangle$ for any $x \in \mathbb{R}^{n}$.


Finally, we will denote by $\mathcal{B}$ the set of functions $\alpha\colon \mathbb{R}\to \mathbb{R}$ mapping bounded sets into bounded sets.

\subsection{Structure of the article}
This article offers a new definition in the framework of observability; nonuniform complete observability, in addition to analyzing the well-known duality it has with controllability. 
In this sense, section 2 details the basic concepts associated with nonuniform complete controllability, whose definition is established in terms of its Gramian matrix and which is strongly related to the concept of nonuniform bounded growth. 

In section 3 we detail the scope of the definition of uniform complete observability from the study of the duality it admits with the nonuniform complete controllability and an example that is related to nonuniform bounded growth.  

On the other hand, Section 4.1 presents the main result of this work, which consists in proving the conservation of nonuniform complete observability by output feedback. Finally, a direct application of duality and the main theorem can be seen in section 4.2; the preservation of complete nonuniform controllability by input feedback is proved.

\section{Controllability: basic notions and nonuniform preliminaries}

As proved in the article \cite{HMR}, the \textit{nonuniform complete controllability} is a concept that is less restrictive than the uniform complete controllability. In that sense, by considering the LTV control system 
\begin{equation}
    \label{LTVControlabilidad}
    \dot{x}(t)=A(t)x(t)+B(t)u(t),
\end{equation}
where $x(t)$, $A(t)$ were described in the subsection 1.3, $t\mapsto B(t)\in M_{n\times p}(\mathbb{R})$ is measurable and bounded on finite intervals, and $u(t)\in\mathbb{R}^{p}$ is the \textit{input}. In addition, the system \eqref{control1a} is known as the \textit{plant} of the system \eqref{LTVControlabilidad}.

We know that if the plant of the system \eqref{LTVControlabilidad} does not satisfy the \textit{Kalman condition} \cite[Prop. 3, ii)]{HMR}, then the uniform complete controllability of the control system is not satisfied (see \cite[Statement (5.19)]{Kalman}). 

For this reason, the concept of nonuniform bounded growth is presented below, which corresponds to a generalization of the Kalman condition and, in a way, is the starting point for defining a new type of controllability that adapts to systems that do not necessarily satisfy the Kalman condition.

\begin{definition} 
The plant of \eqref{LTVControlabilidad} has a nonuniform bounded growth on the interval $J\subset [0,+\infty)$ if there exist constants $K_{0}>0$, $a>0$ and $\varepsilon>0$ such that its transition matrix satisfies
\begin{equation*}
\|\Phi_{A}(t,\tau)\|\leq K_{0}e^{\varepsilon \tau }e^{a|t-\tau|}\quad \textnormal{for any $t,\tau\in J$}.
\end{equation*}
\end{definition}


In view of this, in \cite{HMR} a new type of controllability, called \textit{nonuniform complete controllability}, is established.

\begin{definition}
The linear control system \eqref{control1a} is said to be nonuniformly completely controllable \textnormal{(NUCC)}, if there exist fixed numbers $\mu_0\geq0$, $\mu_1\geq0$, $\tilde{\mu}_0\geq0$, $\tilde{\mu}_1\geq0$ and functions $\alpha_0(\cdot), \beta_{0}(\cdot),\alpha_1(\cdot),\beta_1(\cdot):[0,+\infty)\to(0,+\infty)$ such that for any $t\geq 0$, there exists $\sigma_0(t)>0$ with:
\begin{equation}
\label{alpha0alpha1}
    0<e^{-2\mu_0 t}\alpha_0(\sigma)I\leq W(t,t+\sigma)\leq e^{2\mu_1 t}\alpha_1(\sigma)I,
\end{equation}
\begin{equation}
\label{K}
    0<e^{-2\tilde{\mu}_0 t}\beta_0(\sigma) I\leq K(t,t+\sigma)\leq e^{2\tilde{\mu}_1 t}\beta_1(\sigma) I, 
\end{equation}
for every $\sigma\geq \sigma_{0}(t)$, where 
$$W(t_0,t_1)=\displaystyle\int_{t_0}^{t_1}\Phi_{A}(t_0,s)B(s)B^{T}(s)\Phi_{A}^{T}(t_0,s)\;ds$$
is the \textit{Gramian of controllability associated of the system \eqref{LTVControlabilidad}} and
$$K(t,t+\sigma)=\Phi_{A}(t+\sigma,t)W(t,t+\sigma)\Phi_{A}^{T}(t+\sigma,t).$$
\end{definition}

\begin{remark}
    Based on this definition, we say
without distinction that the system \eqref{LTVControlabilidad} or the pair $(A,B)$ is nonuniformly completely controllable.
\end{remark}

Taking into account the relationship between uniform complete controllability and the Kalman condition \cite{Kalman}, then let us consider the following definition, which will be coupled to the definition of nonuniform complete controllability.

\begin{lemma}
If the control system \eqref{LTVControlabilidad} is nonuniformly completely controllable, then there exist $\nu>0$ and a function  $\alpha(\cdot)\in\mathcal{B}$ satisfying
\begin{equation}
\label{crec-acot}
\|\Phi_{A}(t,\tau)\|\leq e^{\nu\tau}\,\alpha(|t-\tau|)  \quad \textnormal{for any $t,\tau\geq0$}.
\end{equation}
\end{lemma}

\begin{remark}
    As introduced in \cite{HMR}, the expression \eqref{crec-acot} is known as the nonuniform Kalman property. Additionally, we can note that the nonuniform bounded growth is a particular case of the nonuniform Kalman property.
\end{remark}

From the nonuniform complete controllability and the nonuniform Kalman condition, we finish this section by considering the following result about the relationship between these concepts.

\begin{theorem}
\label{p2l3}
Any two of the properties \eqref{alpha0alpha1}, \eqref{K} and \eqref{crec-acot}
imply the third one.
\end{theorem}

\section{Nonuniform Complete Observability; Duality and example}

In this section we will define \textit{nonuniform complete observability} and explore the scope of this definition. Specifically, in subsection 3.1 the duality between this concept and the nonuniform complete controllability will be studied, by identifying the respective dual system of \eqref{control1a}-\eqref{control1b}. On the other hand, subsection 3.2 involves an example where it is studied under which conditions it is possible to have nonuniform complete observability, by pointing out the similarity that this example has with the one shown in \cite[Subsection 3.2]{HMR}.

\subsection{Dual point of view between controllability and observability}

The following concept corresponds to a generalization of the uniform complete observability of the system \eqref{control1a}-\eqref{control1b}, i.e., it will also be described in terms of the observability Gramian detailed in \eqref{gramianobservability}.


\begin{definition}
The system \eqref{control1a}-\eqref{control1b} is nonuniformly completelly observable \textnormal{(NUCO)} if there exist fixed numbers $\nu_0\geq0$, $\nu_1\geq0$ and functions $\vartheta_0$, $\vartheta_1:[0,+\infty)\to (0,+\infty)$ such that for any $t\geq0$, there exists $\sigma_0(t)>0$ with:
\begin{equation}
\label{Mobserv}
\vartheta_0(\sigma)e^{-2\nu_0 t}I\leq M(t,t+\sigma)\leq \vartheta_1(\sigma) e^{2\nu_1 t}
\end{equation}
for every $\sigma\geq\sigma_0(t)$.
\end{definition}

\begin{remark}
The previous definition deserves some comments:
\begin{itemize}
\item[a)] If $\omega:=\max\{\nu_0, \nu_1\}$, the estimate \eqref{Mobserv} can be stated in terms of $\omega$. Also, as done in the dichotomy $\&$ stability literature \cite{Barreira,Chu,Zhou-16}, the numbers $\nu_0$ and $\nu_1$  will be called the gramian nonuniformities of observability.
\item[b)] When $\omega=0$ and $t\mapsto \sigma_0(t)$ is constant for all $t\geq0$, we recover the uniform complete observability definition.  
\item[c)] If \eqref{Mobserv} is satisfied, it is straightforward to verify that NUCO implies complete observability.
\item[d)] In order to facilitate the writing, in some occasions when the context requires it, we will use the notation  $\sigma_{t}:=\sigma_0(t)$.
\end{itemize}
\end{remark}

In order to pave the way for detailing the close relationship between observability and controllability, we consider the dual system of 
\eqref{control1a}-\eqref{control1b} (or $(A,C)$) described by
\begin{equation}
\label{traspuesto1a}
\dot{x}(t)=-A^{T}(t)x(t)+C^{T}(t)u(t), 
\end{equation}
with $x(t)\in\mathbb{R}^{n}$, $u(t)\in\mathbb{R}^{m}$, while the dimensions of $A(t)$ and $C(t)$ are $n\times n$ and $m\times n$ respectively. In addition, the transition matrix of the plant of \eqref{traspuesto1a} is denoted by $\Psi_{-A^{T}}(\cdot,\cdot)$, which satisfied the following relation with the transition matrix of the system \eqref{control1a}:
\begin{equation}
\label{Psi}
    \Psi_{-A^{T}}(t,\tau)=\Phi_{A}^{T}(\tau,t),
\end{equation}
indeed, we have that:
$$\begin{array}{rcl}
\displaystyle\frac{d}{dt}(\Psi_{-A^{T}}(t,\tau))&=&\displaystyle\frac{d}{dt}(\Phi_{A}^{T}(\tau,t))=\frac{d}{dt}(\Phi_{A}(\tau,t))^{T}=-(\Phi_{A}(\tau,t)A(t))^{T},\\
&=&-A^{T}(t)\Phi_{A}^{T}(\tau,t)
=-A^{T}(t)\Psi_{-A^{T}}(t,\tau).
\end{array}$$

The following theorem establishes the relationship between nonuniform complete controllability and nonuniform complete observability by means of the dual system givne by \eqref{traspuesto1a}.

\begin{theorem} 
\label{dualidad}
Assume that the system \eqref{control1a} admits nonuniform bounded growth. 
    The system \eqref{control1a}-\eqref{control1b} is nonuniformly completely observable if and only if the system \eqref{traspuesto1a} is nonuniformly completely controllable. In other words, $(A,C)$ is \textnormal{NUCO} if and only if $(-A^{T},C^{T})$ is \textnormal{NUCC}.
\end{theorem}
\begin{proof}
    The Gramian matrix of observability of the system (\ref{control1a})-(\ref{control1b}) (or $(A,C)$) is given by \eqref{gramianobservability}, with $t_0=t$ and $t_f=t+\sigma$. In addition, by considering \eqref{Psi}, then we have that:
    \begin{equation}
    \label{RelacionObsCont}
    \begin{array}{rcl}
    M(t,t+\sigma)&=&\displaystyle\int_{t}^{t+\sigma} \Phi_{A}^{T}(s,t)C^{T}(s)C(s)\Phi_{A}(s,t)\; ds,\\
    &=&\displaystyle\int_{t}^{t+\sigma} \Psi_{-A^{T}}(t,s)C^{T}(s)C(s)\Psi_{-A^{T}}^{T}(t,s)\; ds,\\
    &=&W(t,t+\sigma),
    \end{array}
    \end{equation}
    where this Gramian matrix of controllability $W(t,t+\sigma)$ corresponds to the system \eqref{traspuesto1a} (or $(-A^{T},C^{T})$). In that sense, if \eqref{Mobserv} is satisfied, then the identity \eqref{RelacionObsCont} allow us to ensure that, for the system $(-A^{T},C^{T})$, the Gramian inequality  \eqref{alpha0alpha1} is verified. In order to obtain the inequality \eqref{K}, namely, to estimate the term $\Psi_{-A^{T}}(t+\sigma,t)W(t,t+\sigma)\Psi_{-A^{T}}^{T}(t+\sigma,t)$, we will verify that the system \eqref{traspuesto1a} admits nonuniform bounded growth. In fact, by considering the relation \eqref{Psi} and by assuming that the system \eqref{control1a} admits nonuniform bounded growth with constants $K_0$, $a$ and $\varepsilon$, then we have that for any $t, \tau\in\mathbb{R}_{0}^{+}$:
    $$
    \begin{array}{rcl}
    \left \| \Psi_{-A^{T}}(t,\tau)\right \|&=&\left \| \Phi_{A}^{T}(\tau,t)\right\|=\left \| \Phi_{A}(\tau,t)\right \|,\\
    &\leq &K_0e^{a|\tau-t|+\varepsilon t}=K_0e^{a|\tau-t|+\varepsilon ((t-\tau)+\tau)},\\
    &\leq& K_0 e^{(a+\varepsilon)|t-\tau|+\varepsilon \tau}
    \end{array}
    $$
    i.e, the plant of system \eqref{traspuesto1a} verified the nonuniform bounded growth property with constants $K_0$, $a+\varepsilon$ and $\varepsilon$. Therefore, by considering Theorem \ref{p2l3}, we can ensure the estimation \eqref{K} for this system and we conclude that the system \eqref{traspuesto1a} is nonuniformly completely controllable.

    On the other hand, if the system \eqref{traspuesto1a} (or $(-A^{T},C^{T})$) is nonuniformly completely controllable, then the inequality \eqref{alpha0alpha1}, where $W(t,t+\sigma)$ corresponds to the Gramian controllability of the system $(-A^{T},C^{T})$ (see \eqref{RelacionObsCont}), is satisfied. In addition, by considering \eqref{RelacionObsCont}, it is straightforward to conclude that \eqref{Mobserv} is verified and therefore, the system \eqref{control1a}-\eqref{control1b} (or $(A,C)$) is nonuniformly completely observable.
\end{proof}

\subsection{An example of a nonuniformly completely observable system} Based on the duality offered by the previous theorem, this example is strongly focused on the nonuniform bounded growth property and follows the same line of the example presented in \cite{HMR}.
If we assume that the linear system (\ref{control1a}) admits the property of nonuniform bounded growth with parameters $K_0$, $a$ and $\varepsilon$, it is interesting to explore some conditions on $C(t)$ leading to the estimate (\ref{Mobserv}).
In order to get it, we will assume that there exist $\gamma_0$, $\gamma_1>0$ and $c_0$, $c_1>0$ such that for any $t\geq0$: 
    $$c_0e^{-2\gamma_0 t}I\leq C^{T}(t)C(t)\leq c_1e^{2\gamma_1 t}I.$$
    
    Note that the previous condition includes the particular case of constant $C$, with $C^{T}C>0$. 
    Based on the transition matrix property, we have that for any vector $x\neq0$, it is verified that  
$
x=\Phi_{A}^{T}(s,t)\Phi_{A}^{T}(t,s)x
$
and then:    
    $$|x|^{2}\leq\left\|\Phi_{A}^{T}(s,t)\right\|^{2}|\Phi_{A}^{T}(t,s)x|^{2}\Rightarrow \frac{|x|^{2}}{\left\|\Phi_{A}(s,t)\right\|^{2}}\leq|\Phi_{A}^{T}(t,s)x|^{2}, $$
    and by considering the nonuniform bounded growth hypothesis, we can ensure that
    \begin{equation*}
    \left\|\Phi_{A}(s,t)\right\|\leq K_0e^{a|s-t|+\varepsilon t}\Leftrightarrow\frac{1}{K_0e^{a|s-t|+\varepsilon t}}\leq \frac{1}{\left\|\Phi_{A}(s,t)\right\|}.
    \end{equation*}

    From the above, we will obtain the right-hand estimate of (\ref{alpha0alpha1}). For any $t\geq0$, there exist $\sigma_0(t)>0$ such that for any $\sigma\geq\sigma_0(t)$ and $x\neq0$:
    \begin{displaymath}
    \begin{array}{rcl}
    \displaystyle\int_{t}^{t+\sigma}x^{T}\Phi_{A}^{T}(t,s)C^{T}(s)C(s)\Phi_{A}(t,s)x\;ds&\leq& \displaystyle\int_{t}^{t+\sigma}K_0^{2}e^{2a(s-t)+2\varepsilon s}c_1^{2}e^{2\gamma_1 s}|x|^{2}\;ds\\\\
    &=&\displaystyle K_0^{2}c_1^{2}|x|^{2}e^{-2at}\int_{t}^{t+\sigma}e^{2(a+\varepsilon+\gamma_1)s}\;ds,
    \end{array}
    \end{displaymath}
which implies that
 \begin{displaymath}
    \begin{array}{rcl}
    x^{T}M(t,t+\sigma)x&\leq& 
    \displaystyle \frac{K_0^{2}c_1^{2}}{2(a+\varepsilon+\gamma_1)}e^{-2at}\left [e^{2(a+\varepsilon+\gamma_1)(t+\sigma)}-e^{2(a+\varepsilon+\gamma_1)t}\right]|x|^{2},\\\\
    &=&\displaystyle \frac{K_0^{2}c_1^{2}}{2(a+\varepsilon+\gamma_1)}\left [e^{2(a+\varepsilon+\gamma_1)\sigma}-1\right]e^{2(\varepsilon+\gamma_1)t}|x|^{2}.\\\\
    \end{array}
    \end{displaymath}

Similarly, notice that
    $$\begin{array}{rcl}
         \displaystyle\int_{t}^{t+\sigma}x^{T}\Phi_{A}^{T}(t,s)C^{T}(s)C(s)\Phi_{A}(t,s)x\;ds &\geq&  \displaystyle\int_{t}^{t+\sigma}\frac{1}{K_0^{2}}e^{-2a(s-t)-2\varepsilon t}c_0^{2}e^{-2\gamma_0 s}|x|^{2}\;ds,\\\\
         &=&\displaystyle\frac{c_0^{2}}{K_0^{2}}|x|^{2}e^{2(a-\varepsilon)t}\int_{t}^{t+\sigma}e^{-2(a+\gamma_0)s}\;ds,
    \end{array}$$
which implies
\begin{displaymath}
    \begin{array}{rcl}
         x^{T}M(t,t+\sigma)x &\geq&  \displaystyle\frac{c_0^{2}}{K_0^{2}2(a+\beta_0)}e^{2(a-\eta)t}\left[e^{-2(a+\beta_0)t}-e^{-2(a+\beta_0)(t+\sigma)}\right]|x|^{2}\\\\
         &=&\displaystyle \frac{c_0^{2}}{K_0^{2}2(a+\gamma_0)}\left[1-e^{-2(a+\gamma_0)\sigma}\right]e^{-2(\varepsilon+\gamma_0)t}|x|^{2}
    \end{array}
\end{displaymath}
and the nonuniform complete observability follows by considering 
$$\vartheta_{0}(\sigma)=\displaystyle \frac{c_0^{2}}{K_0^{2}2(a+\gamma_0)}\left[1-e^{-2(a+\gamma_0)\sigma}\right]\;\textnormal{and}\;\vartheta_1(\sigma)=\displaystyle \frac{K_0^{2}c_1^{2}}{2(a+\varepsilon+\gamma_1)}\left [e^{2(a+\varepsilon+\gamma_1)\sigma}-1\right].$$

As an extra comment, in the previous proof we can notice that $\sigma_0(t)=\sigma_0>0$ is constant and works in order to prove the inequalities for all $t\geq0$.

\section{Observability conservation by output feedback in a nonuniform context} 
\subsection{Main result} In this subsection we will study the invariance of the nonuniform complete observability when we consider an output feedback, which corresponds to the main result of this work. The ideas and results that we will see below are strongly based on \cite{Zhang2015ObservabilityCB}, but they are adapted to the nonuniform notion of the concepts of bounded growth and complete observability and their respective bounds.

To begin, consider the following LTV system
\begin{subequations}
  \begin{empheq}[left=\empheqlbrace]{align}
    & \dot{x}(t)=A(t)x(t)+B(t)u(t), \label{LTVsystem1a} \\
    & y(t)=C(t)x(t), \label{LTVsystem1b}
  \end{empheq}
\end{subequations}
where $A(t)\in M_{n\times n}(\mathbb{R})$, $B(t)\in M_{n\times p}(\mathbb{R})$ and $C(t)\in M_{m\times n}(\mathbb{R})$. Moreover, by using an output feedback $u(t)=-F(t)y(t)$, for some matrix $F(\cdot)\in M_{p\times m}(\mathbb{R})$, and by defining $K(t):=B(t)F(t)$, then we replace this output feedback in \eqref{LTVsystem1a}-\eqref{LTVsystem1b} and we obtain the following two systems:

$$
\left\{\begin{array}{lcl}
\dot{x}(t)&=&A(t)x(t)-K(t)y(t),\\
y(t)&=&C(t)x(t),
\end{array}
\right.
$$
and
\begin{subequations}
  \begin{empheq}[left=\empheqlbrace]{align}
    & \dot{x}(t)=(A(t)-K(t)C(t))x(t), \label{ClosedLoop2a} \\
    & y(t)=C(t)x(t), \label{ClosedLoop2b}
  \end{empheq}
\end{subequations}
which are equivalents. In this way, the systems \eqref{ClosedLoop2a}-\eqref{ClosedLoop2b}, and additionally \eqref{ClosedLoop2a}-\eqref{ClosedLoop2b}, have the same property of observability than \eqref{LTVsystem1a}-\eqref{LTVsystem1b}.

In order to pave the way towards the main result of this section, we will show the following technical result, which ensure that the nonuniform bounded growth property is preserved in a closed loop system.

\begin{lemma}
\label{RobustezNUBG}
    Suppose that the plant of system \eqref{LTVsystem1a} admits nonuniform bounded growth with constants $K_0$, $a$ and $\varepsilon$. In addition, by considering the system \eqref{ClosedLoop2a}, suppose that there exist $\delta>0$ and $\gamma>0$ such that:
    \begin{equation}
    \label{Kcondition}
    \|K(z)\|\leq \mathcal{K}e^{\delta z},
    \end{equation}
    \begin{equation} 
    \label{Ccondition}
    \|C(z)\|\leq \mathcal{C}e^{-\gamma z}.
    \end{equation}
    If we have the following condition about the constants
    \begin{equation}
        \label{conditionGED}
    \gamma-(\varepsilon+\delta)>0,
    \end{equation}
    then the system \eqref{ClosedLoop2a} admits nonuniform bounded growth property.
\end{lemma}
\begin{proof}
    The transition matrix of the system \eqref{ClosedLoop2a}, where 
    \begin{equation}
        \label{ATilde}
        \tilde{A}(t):=A(t)-K(t)C(t),
    \end{equation}
    is given by 
    \begin{equation} 
    \label{SolucionA-KC}
    \Phi_{\tilde{A}}(t_2,t_1):=\Phi_{A}(t_2,t_1)-\int_{t_1}^{t_2}\Phi_{A}(t_2,s)K(s)C(s)\Phi_{\tilde{A}}(s,t_1)\; ds
    \end{equation}
    and by assuming the nonuniform bounded growth property of the plant of system \eqref{LTVsystem1a}, we have the following estimate:
    $$
    \begin{array}{rcl}
    \left\|\Phi_{\tilde{A}}(t_2,t_1)\right\|&\leq&\displaystyle K_0e^{a(t_2-t_1)+\varepsilon t_1}+\int_{t_1}^{t_2}K_0e^{a(t_2-s)+\varepsilon s}\mathcal{K}e^{\delta s}\mathcal{C}e^{-\gamma s}\left\|\Phi_{\tilde{A}}(s,t_1)\right\|\; ds,\\\\
    e^{-at_2}\left\|\Phi_{\tilde{A}}(t_2,t_1)\right\|&\leq&\displaystyle K_0e^{-at_1+\varepsilon t_1}+\int_{t_1}^{t_2}K_0e^{-as}\mathcal{K}\mathcal{C}e^{(\varepsilon+\delta-\gamma)s}\left\|\Phi_{\tilde{A}}(s,t_1)\right\|\; ds,\\\\
    \end{array}
    $$
and by considering the Gronwall inequality and the condition \eqref{conditionGED}, we obtain that:
$$\begin{array}{rcl}
e^{-at_2}\left\|\Phi_{\tilde{A}}(t_2,t_1)\right\|&\leq&\displaystyle K_0e^{-at_1+\varepsilon t_1}e^{\int_{t_1}^{t_2}K_0\mathcal{K}\mathcal{C}e^{(\varepsilon+\delta-\gamma)s}\; ds},\\
\left\|\Phi_{\tilde{A}}(t_2,t_1)\right\|&\leq&\displaystyle K_0e^{\frac{K_0\mathcal{K}\mathcal{C}}{\gamma-(\varepsilon+\delta)}}e^{a(t_2-t_1)+\varepsilon t_1},
\end{array}$$
which prove that the system \eqref{ClosedLoop2a} admits nonuniform bounded growth.
\end{proof}

The following theorem establishes the preservation of nonuniform complete observability by output feedback, whose proof follows the main result of \cite{Zhang2015ObservabilityCB} (see \cite{f3ec37272184481db229561565c2d5fb} for another similar result).

\begin{theorem}
\label{observableconK}
    Suppose that the plant of the system \eqref{LTVsystem1a} admits nonuniform bounded growth with constants $K_0$, $a$ and $\varepsilon$. By assuming that the conditions \eqref{Kcondition}, \eqref{Ccondition}, \eqref{conditionGED} and 
    \begin{equation}
        \label{conditionAED}
        -a+\varepsilon+\delta>0
    \end{equation}
    are satisfied, then the system \eqref{control1a}-\eqref{control1b} (or $(A,C)$) is nonuniformly completely observable if and only if the system \eqref{ClosedLoop2a}-\eqref{ClosedLoop2b} (or $(A-KC,C)$) is nonuniformly completely observable.
\end{theorem} 

\begin{remark}
    Before proceeding with the proof of the theorem, we can note that there exists a kind of ``symmetry" between the matrix $A(t)$ and $\tilde{A}(t)$ described in \eqref{ATilde},    
    this is due to the fact that \eqref{ATilde} can be rewrite as follows
    \begin{equation*}
        A(t)=\tilde{A}(t)-(-K(t))C(t).
    \end{equation*}

    Based on the above, the implication from right to left on the Theorem \textnormal{\ref{observableconK}} can be proved based on the implication from left to right.
\end{remark}

\begin{proof}
    By considering the parameter $(s,t)$ instead of $(t_2,t_1)$ in \eqref{SolucionA-KC}, we have that
    $$\Phi_{\tilde{A}}(s,t)-\Phi_{A}(s,t)=-\displaystyle\int_{t}^{s}\Phi_{A}(s,p)K(p)C(p)\Phi_{\tilde{A}}(p,t)\; dp.$$

    Multiplying from the left on both sides of the previous equality by $C(s)$, multiplying from the right by any unit vector $v\in\mathbb{R}^{n}$ and by taking the integral with respect to $s$, of the squared norm of both sides, then
    $$\displaystyle \int_{t}^{t+\sigma}\left \| C(s)(\Phi_{\tilde{A}}(s,t)-\Phi_{A}(s,t))v\right \|^{2}ds=\displaystyle \int_{t}^{t+\sigma}\left \| C(s)\int_{t}^{s}\Phi_{A}(s,p)K(p)C(p)\Phi_{\tilde{A}}(p,t)v\;dp\right \|^{2}ds.$$

    By developing the squared Euclidean norm at the left hand side and by reordering the terms, we have that
    \begin{equation}
    \label{igualdadinicial}
    \begin{array}{rcl}
    \displaystyle \int_{t}^{t+\sigma}\left \| C(s)\Phi_{A}(s,t)v\right \|^{2}ds&=&-\displaystyle \int_{t}^{t+\sigma}\left \| C(s)\Phi_{\tilde{A}}(s,t)v\right \|^{2}ds\\
    &+&\displaystyle\int_{t}^{t+\sigma}2v^{T}\Phi_{A}^{T}(s,t)C^{T}(s)C(s)\Phi_{\tilde{A}}(s,t)v\;ds\\
    &+& \displaystyle \int_{t}^{t+\sigma}\left \| C(s)\int_{t}^{s}\Phi_{A}(s,p)K(p)C(p)\Phi_{\tilde{A}}(p,t)vdp\right \|^{2} ds.
    \end{array}\end{equation}

    Now, each terms described in the previous equality will be properly bounded in order to achieve the nonuniform complete observability of the system \eqref{ClosedLoop2a}-\eqref{ClosedLoop2b}. 

    Firstly, if the system \eqref{control1a}-\eqref{control1b} is nonuniformly completely observable, from \eqref{gramianobservability} we have that for any unitary vector $v\in\mathbb{R}^{n}$:
    \begin{equation} 
    \label{MTeorema}
    \begin{array}{rcl}
    \vartheta_0(\sigma)e^{-2\nu_0 t}\leq v^{T}M(t,t+\sigma)v&=&\displaystyle\int_{t}^{t+\sigma}\left \|C(s)\Phi_{A}(s,t)v \right \|^{2}ds,\\\\
    v^{T}M(t,t+\sigma)v&\leq& v^{T}\vartheta_1(\sigma)e^{2\nu_1 t}v=\vartheta_1(\sigma)e^{2\nu_1 t}.
    \end{array}
    \end{equation}

From this point forward, define $\Lambda$ as the term that we have to bound as follows:
    \begin{equation} 
    \label{lambda}
    \Lambda:=\displaystyle\int_{t}^{t+\sigma}\left\|C(s)\Phi_{\tilde{A}}(s,t)v \right \|^{2}ds=v^{T}M(t,t+\sigma)v.
    \end{equation}

    On the other hand, by using the Cauchy-Schwarz inequality and by considering \eqref{MTeorema} and \eqref{lambda} we have that:
    \begin{equation} \label{CSdesigualdad}
    \begin{array}{rcl}&&\displaystyle\int_{t}^{t+\sigma}2v^{T}\Phi_{A}^{T}(s,t)C^{T}(s)C(s)\Phi_{\tilde{A}}(s,t)v\;ds\\
    &\leq& \displaystyle 2\sqrt{\int_{t}^{t+\sigma}\left\|C(s)\Phi_{A}(s,t)v \right \|^{2}ds\int_{t}^{t+\sigma}\left\|C(s)\Phi_{\tilde{A}}(s,t)v \right \|^{2}ds}\\
    &\leq& 2\sqrt{\vartheta_1(\sigma)}e^{\nu_1 t}\sqrt{\Lambda}
    \end{array}
    \end{equation}

After that, by using \eqref{conditionAED} we have that:
$$\begin{array}{rcl}
\displaystyle \int _{t}^{s}\left\| \Phi_{A}(s,p)\right \|^{2}\left\| K(p)\right\|^{2}dp&\leq& \displaystyle\int_{t}^{s}K_0e^{2a(s-p)+2\varepsilon p}\mathcal{K}e^{2\delta p}dp,\\
&\leq&\displaystyle \frac{K_0\mathcal{K}}{2(-a+\varepsilon+\delta)}e^{2as}\left [ e^{2(-a+\varepsilon+\delta)s}-e^{2(-a+\varepsilon+\delta)t}\right ],\\
&\leq&\displaystyle \frac{K_0\mathcal{K}}{2(-a+\varepsilon+\delta)}e^{2(\delta+\varepsilon)s}.
\end{array}$$

Therefore, by mixing the previous estimates we obtain that:
$$\begin{array}{rcl}
&&\displaystyle \int_{t}^{t+\sigma}\left \| C(s)\int_{t}^{s}\Phi_{A}(s,p)K(p)C(p)\Phi_{\tilde{A}}(p,t)v\;dp\right \|^{2} ds\\
&\leq&\displaystyle\int_{t}^{t+\sigma}\mathcal{C}e^{-2\gamma s}\left(\int_{t}^{s}\left\|\Phi_{A}(s,p) \right \|^{2}\left \|K(p) \right \|^{2}dp \right)\left (\int_{t}^{s}\left\| C(p)\Phi_{\tilde{A}}(p,t)v\right\|^{2}dp \right)ds,\\
&\leq&\Lambda\displaystyle\int_{t}^{t+\sigma}\frac{K_0\mathcal{K}\mathcal{C}}{2(-a+\varepsilon+\delta)}e^{-2(\gamma-(\varepsilon+\delta))s}ds,\\
&=&\displaystyle\Lambda\frac{K_0\mathcal{K}\mathcal{C}}{4(-a+\varepsilon+\delta)(\gamma-(\varepsilon+\delta))}e^{-2(\gamma-(\varepsilon+\delta))t}\left[1-e^{-2(\gamma-(\varepsilon+\delta))\sigma}\right],\\
&\leq&\displaystyle\Lambda\frac{K_0\mathcal{K}\mathcal{C}}{4(-a+\varepsilon+\delta)(\gamma-(\varepsilon+\delta))}.
\end{array}$$

Based on the previous estimations, 
\eqref{igualdadinicial} is expressed in terms of the following inequality:

\begin{equation}
\label{desigualdadlambda}
\vartheta_0(\sigma)e^{-2\nu_0 t}\leq -\Lambda+2\sqrt{\vartheta_1(\sigma)}e^{\nu_1 t}\sqrt{\Lambda}+\varphi\Lambda=(\varphi-1)\Lambda+2\sqrt{\vartheta_1(\sigma)}e^{\nu_1 t}\sqrt{\Lambda}
\end{equation}
where 
$$\varphi=\frac{K_0\mathcal{K}\mathcal{C}}{4(-a+\varepsilon+\delta)(\gamma-(\varepsilon+\delta))}.
$$

We will analyze two cases: 
\begin{itemize}
    \item [a)] If $\varphi=1$, then from \eqref{desigualdadlambda} we have that:
$$
\vartheta_0(\sigma)e^{-2\nu_0 t}\leq 2\sqrt{\vartheta_1(\sigma)}e^{\nu_1 t}\sqrt{\Lambda}\;\Longrightarrow\;
\displaystyle\frac{\vartheta^{2}_0(\sigma)}{4\vartheta_1(\sigma)}e^{-(4\nu_0+2\nu_1)t}\leq\Lambda.
$$

\item[b)] If $\varphi\neq1$, then we rewrite \eqref{desigualdadlambda} as follows:
$$(\varphi-1)\Lambda+2\sqrt{\vartheta_1(\sigma)}e^{\nu_1 t}\sqrt{\Lambda}-\vartheta_0(\sigma)e^{-2\nu_0 t}\geq0,$$
and solving this inequality we obtain that:
$$\begin{array}{rcl}
\displaystyle\left |\frac{-\sqrt{\vartheta_1(\sigma)e^{2\nu_1 t}}+\sqrt{\vartheta_1(\sigma)e^{2\nu_1 t}+(\varphi-1)\vartheta_0(\sigma)e^{-2\nu_0 t}}}{\varphi-1}\right|&\leq&\sqrt{\Lambda},\\
\displaystyle \left(\frac{(\varphi-1)\vartheta_0(\sigma)e^{-2\nu_0 t}}{(\varphi-1)\left(\sqrt{\vartheta_1(\sigma)e^{2\nu_1 t}}+\sqrt{\vartheta_1(\sigma)e^{2\nu_1 t}+(\varphi-1)\vartheta_0(\sigma)e^{-2\nu_0 t}}\right)}\right)^{2}&\leq&\Lambda,\\
\displaystyle \left(\frac{\vartheta_0(\sigma)e^{-2\nu_0 t}}{\sqrt{\vartheta_1(\sigma)e^{2\nu_1 t}}+\sqrt{\vartheta_1(\sigma)e^{2\nu_1 t}+(\varphi-1)\vartheta_0(\sigma)e^{-2\nu_0 t}}}\right)^{2}&\leq&\Lambda.
\end{array}$$
Two subcases arise from this:
\end{itemize}
\begin{itemize}
    \item [b.i)] If $\varphi-1>0$, we obtain that:
    $$\begin{array}{rcl}
     \displaystyle \left(\frac{\vartheta_0(\sigma)e^{-2\nu_0 t}}{2\sqrt{\max\{\vartheta_1(\sigma),(\varphi-1)\vartheta_0(\sigma)\}e^{2\nu_1 t}}}\right)^{2}&\leq&\Lambda,\\
\displaystyle\frac{\vartheta_0^{2}(\sigma)}{4\max\{\vartheta_1(\sigma),(\varphi-1)\vartheta_0(\sigma)\}}e^{-(4\nu_0+2\nu_1)t}&\leq&\Lambda.
\end{array}$$
    \item[b.ii)] If $\varphi-1<0$, then we have that:
    $$
    \displaystyle\left ( \frac{\vartheta_0(\sigma)e^{-2\nu_0 t}}{2\sqrt{\vartheta_1(\sigma)e^{2\nu_1 t}}}\right )^{2}\leq\Lambda\;\Longrightarrow\;
    \displaystyle\frac{\vartheta^{2}_0(\sigma)}{4\vartheta_1(\sigma)}e^{-(4\nu_0+2\nu_1)t}\leq\Lambda.
    $$
\end{itemize}

Now we must establish the upper bound for $\Lambda$, which will be obtained from an argument similar to the one used to get the lower bound. Based on this kind of symmetry existing between $A(t)$ and $\tilde{A}(t)$, the ``symmetric'' counterpart of (\ref{igualdadinicial}) corresponds to exchange $\Phi_{A}(\cdot,\cdot)$ with $\Phi_{\tilde{A}}(\cdot,\cdot)$ and replacing $K(p)$ by $-K(p)$. In this way, we will have the following expression: 
\begin{equation*}
    \begin{array}{rcl}
    \displaystyle \int_{t}^{t+\sigma}\left \| C(s)\Phi_{\tilde{A}}(s,t)v\right \|^{2}ds&=&-\displaystyle \int_{t}^{t+\sigma}\left \| C(s)\Phi_{A}(s,t)v\right \|^{2}ds\\
    &+&\displaystyle\int_{t}^{t+\sigma}2v^{T}\Phi_{\tilde{A}}^{T}(s,t)C^{T}(s)C(s)\Phi_{A}(s,t)v\;ds\\
    &+& \displaystyle \int_{t}^{t+\sigma}\left \| C(s)\int_{t}^{s}\Phi_{\tilde{A}}(s,p)K(p)C(p)\Phi_{A}(p,t)v\;dp\right \|^{2} ds.
    \end{array}\end{equation*}

By rewriting the previous equality, we obtain that:
\begin{equation}
    \label{igualdadinicialcaso2}
    \begin{array}{rcl}
\displaystyle\int_{t}^{t+\sigma}\left \| C(s)\Phi_{A}(s,t)v\right \|^{2}ds&=&-\displaystyle \int_{t}^{t+\sigma}\left \| C(s)\Phi_{\tilde{A}}(s,t)v\right \|^{2}ds\\
&+&\displaystyle\int_{t}^{t+\sigma}2v^{T}\Phi_{\tilde{A}}^{T}(s,t)C^{T}(s)C(s)\Phi_{A}(s,t)v\;ds\\
    &+& \displaystyle \int_{t}^{t+\sigma}\left \| C(s)\int_{t}^{s}\Phi_{\tilde{A}}(s,p)K(p)C(p)\Phi_{A}(p,t)v\;dp\right \|^{2} ds.
    \end{array}\end{equation}

The second term on the right-hand side of \eqref{igualdadinicialcaso2} will be bounded by using the Cauchy-Schwarz inequality as we did it in \eqref{CSdesigualdad}. On the other hand, by considering Lemma \ref{RobustezNUBG}, \eqref{Kcondition} and \eqref{conditionAED} we have that:
$$\begin{array}{rcl}
\displaystyle \int _{t}^{s}\left\| \Phi_{\tilde{A}}(s,p)\right \|^{2}\left\| K(p)\right\|^{2}dp&\leq& \displaystyle\int_{t}^{s}K_0e^{\frac{K_0\mathcal{K}\mathcal{C}}{\gamma-(\varepsilon+\delta)}}e^{2a(s-p)+2\varepsilon p}\mathcal{K}e^{2\delta p}dp,\\
&\leq&\displaystyle \frac{K_0e^{\frac{K_0\mathcal{K}\mathcal{C}}{\gamma-(\varepsilon+\delta)}}\mathcal{K}}{2(-a+\varepsilon+\delta)}e^{2as}\left [ e^{2(-a+\varepsilon+\delta)s}-e^{2(-a+\varepsilon+\delta)t}\right ],\\
&\leq&\displaystyle \frac{K_0e^{\frac{K_0\mathcal{K}\mathcal{C}}{\gamma-(\varepsilon+\delta)}}\mathcal{K}}{2(-a+\varepsilon+\delta)}e^{2(\delta+\varepsilon)s}.
\end{array}$$

Therefore, by mixing the previous estimates:
$$\begin{array}{rcl}
&&\displaystyle \int_{t}^{t+\sigma}\left \| C(s)\int_{t}^{s}\Phi_{\tilde{A}}(s,p)K(p)C(p)\Phi_{A}(p,t)v\;dp\right \|^{2} ds\\
&\leq&\displaystyle\int_{t}^{t+\sigma}\mathcal{C}e^{-2\gamma s}\left(\int_{t}^{s}\left\|\Phi_{\tilde{A}}(s,p) \right \|^{2}\left \|K(p) \right \|^{2}dp \right)\left (\int_{t}^{s}\left\| C(p)\Phi_{A}(p,t)v\right\|^{2}dp \right)ds,\\
&\leq&\displaystyle\int_{t}^{t+\sigma}\frac{K_0e^{\frac{K_0\mathcal{K}\mathcal{C}}{\gamma-(\varepsilon+\delta)}}\mathcal{K}\mathcal{C}}{2(-a+\varepsilon+\delta)}e^{-2(\gamma-(\varepsilon+\delta))s}\vartheta_1(\sigma)e^{2\nu_1 t}ds,\\
&=&\displaystyle\frac{\vartheta_1(\sigma)K_0e^{\frac{K_0\mathcal{K}\mathcal{C}}{\gamma-(\varepsilon+\delta)}}\mathcal{K}\mathcal{C}}{4(-a+\varepsilon+\delta)(\gamma-(\varepsilon+\delta))}e^{2\nu_1 t}e^{-2(\gamma-(\delta+\varepsilon))t}\left[1-e^{-2(\gamma-(\varepsilon+\delta))\sigma}\right],\\
&\leq&\displaystyle\frac{\vartheta_1(\sigma)K_0e^{\frac{K_0\mathcal{K}\mathcal{C}}{\gamma-(\varepsilon+\delta)}}\mathcal{K}\mathcal{C}}{4(-a+\varepsilon+\delta)(\gamma-(\varepsilon+\delta))}e^{2\nu_1 t}.
\end{array}$$

On this way, we have the following:
$$\vartheta_0(\sigma)e^{-2\nu_0 t}\leq -\Lambda+2\sqrt{\vartheta_1(\sigma)}e^{\nu_1 t}\sqrt{\Lambda}+\psi e^{2\nu_1 t}$$
where 
$$\psi=\displaystyle\frac{\vartheta_1(\sigma)K_0e^{\frac{K_0\mathcal{K}\mathcal{C}}{\gamma-(\varepsilon+\delta)}}\mathcal{K}\mathcal{C}}{4(-a+\varepsilon+\delta)(\gamma-(\varepsilon+\delta))},$$
then we have that
$$\begin{array}{rcl}
     \sqrt{\Lambda}&\leq&\sqrt{\vartheta_1(\sigma)e^{2\nu_1 t}}+\sqrt{\vartheta_1(\sigma)e^{2\nu_1 t}+(\psi e^{2\nu_1 t}-\vartheta_0(\sigma)e^{-2\nu_0 t})},  \\
     \sqrt{\Lambda}&\leq&2 \sqrt{(\vartheta_1(\sigma)+\psi) e^{2\nu_1 t}},\\
     \Lambda&\leq&4(\vartheta_1(\sigma)+\psi) e^{2\nu_1 t}.
\end{array}$$

In summary, the following expressions correspond to the estimates of $\Lambda$: if $\varphi-1\leq0$ we have that
$$\displaystyle\frac{\vartheta^{2}_0(\sigma)}{4\vartheta_1(\sigma)}e^{-(4\nu_0+2\nu_1)t}\leq \Lambda\leq4(\vartheta_1(\sigma)+\psi) e^{2\nu_1 t}$$
and if $\varphi-1>0$ we obtain that:
$$\displaystyle\frac{\vartheta_0^{2}(\sigma)}{4\max\{\vartheta_1(\sigma),(\varphi-1)\vartheta_0(\sigma)\}}e^{-(4\nu_0+2\nu_1)t}\leq \Lambda\leq4(\vartheta_1(\sigma)+\psi) e^{2\nu_1 t},$$
which proves the nonuniform complete observability of the system \eqref{ClosedLoop2a}-\eqref{ClosedLoop2b}.
\end{proof}

\subsection{Consequence of duality; preservation of the nonuniform complete controllability}

This subsection will present an important consequence of the proved duality between nonuniform complete controllability and nonuniform complete observability. Specifically, based on the Theorem \ref{observableconK}, the preservation of nonuniform complete controllability via input feedback will be established.

Additionally, in order to make the following result and its proof even clearer, we will use letter notation to indicate the systems we will be working with.

\begin{theorem}
Assume that the plant of the system $(A,B)$ admits nonuniform bounded growth with constants $K_0$, $a$ and $\varepsilon$. If the following estimates are satisfied:
$$\|L^{T}(z)\|\leq\mathcal{L}e^{\ell z},$$
$$\|B^{T}(z)\|\leq\mathcal{B}e^{-\beta z},$$
$$\beta-(\varepsilon+\ell)>0,$$
the system $(A,B)$ is nonuniformly completely controllable if and only if the system $(A-BL,B)$ is nonuniformly completely controllable.    
\end{theorem}

\begin{proof}
    By using the Theorem \ref{dualidad}, the system $(A,B)$ is nonuniformly completely controllable if and only if $(-A^{T},B^{T})$ is nonuniformly completely observable. On the other hand, by Theorem \ref{observableconK}, the system $(-A^{T}+L^{T}B^{T},B^{T})$ is also nonuniformly completely observable. Finally, by using again the Theorem \ref{dualidad}, we can conclude that $(A-BL,B)$ is nonuniformly completely controllable.
\end{proof}

\section{Conclusions and comments}


This work introduced a new property of observability for linear time-varying systems: the nonuniform complete observability (NUCO), which is more general than the uniform complete observability (UCO) but is more specific than the complete observability (CC), both classical in control theory. 

\medskip

First, we revisit the duality existing between controllability and observability in a nonuniform framework. This paper also proved that the nonuniformly complete observability is preseved via output feedback. In addition, we point out the fact that from the established duality, it is possible to move from observability to controllability and thus prove that controllability is preserved via input feedback.

\medskip

Moreover, based on the stabilization result obtained in \cite{HMR}, the duality proved in this work would allow us to move from nonuniform complete observability to the concept of nonuniform complete stabilizability defined in that work. 

\medskip
The current results can certainly be improved and also raise new questions:

\medskip

\noindent a) The concept of nonuniform bounded growth is very important to prove the main result of this work, therefore it would be very interesting to generalize our result by considering the nonuniform Kalman condition instead of the Kalman condition or uniform bounded growth (the equivalence was proved in \cite[Lemma 6]{HMR}).

\medskip 

\noindent b) In the same way that it is possible to relate the controllability and the stabilizability of a system; in a framework of LTV systems see more details in \cite{Ikeda} for the uniform case and \cite{HMR} for the nonuniform context; it is possible to study the correspondence between observability and detectability (see details in \cite{Tranninger} for the study in the uniform case for LTV systems). 

\medskip

In simple terms, considering the LTV system \eqref{control1a}-\eqref{control1b}, we seek to guarantee the existence of an observer described by the equation
\begin{equation*}
\dot{\hat{x}}(t)=A(t)\hat{x}(t)+P(t)[y(t)-C(t)\hat{x}(t)],
\end{equation*}
such that the linear system
\begin{equation}
        \label{SistError}
        \dot{e}(t)=[A(t)-P(t)C(t)]e(t)
    \end{equation}
with $e(t)=x(t)-\dot{\hat{x}}(t)$, will be uniformly exponentially stable. 

Formally, we have the following definition of uniform exponential detectability (see \cite{RAVI1992455,Tranninger}):
\begin{definition}
    The system \eqref{control1a}-\eqref{control1b} $(A,C)$ is called uniformly exponentially detectable if there exists a uniformly bounded output injection gain $P(t)$ such that the system \eqref{SistError} is uniformly exponentially stable. 
\end{definition}

Based on the above and considering the nonuniform exponential stabilizability defined in \cite{HMR}, it will be interesting to define the nonuniform exponential detectability and thus complement it with the concept of nonuniform complete observability.

\bibliographystyle{abbrv}

\bibliography{samplebib}

\end{document}